\documentclass{amsart}
\usepackage{latexsym, amssymb, amsmath, amsthm}
\usepackage[T1]{fontenc}
\usepackage{lmodern}
\usepackage{enumerate}
\usepackage{mathabx}
\usepackage{csquotes}
\usepackage[mathscr]{euscript}
\usepackage{parskip}
\usepackage{thmtools, thm-restate}
\usepackage{bussproofs}
\usepackage{xcolor}

\usepackage{bm}

\newcommand{\defiff}{\stackrel{\mbox{\scriptsize $\textrm{def}$}}{\iff}}

\makeatletter
\newlength\knuthian@fdfive
\def\mathpal@save#1{\let\was@math@style=#1\relax}
\def\utilde#1{\mathpalette\mathpal@save
              {\setbox124=\hbox{$\was@math@style#1$}%
\setbox125=\hbox{$\fam=3\global\knuthian@fdfive=\fontdimen5\font$}
\setbox125=\hbox{$\widetilde{\vrule height 0pt depth 0pt width \wd124}$}%
               \baselineskip=1pt\relax
               \lineskiplimit=\z@\relax
               \lineskip=1pt\relax
               \vtop{\copy124\copy125\vskip -\knuthian@fdfive}}}
\makeatother

\usepackage{pifont}

\declaretheorem[numberwithin=section]{theorem}
\newtheorem{lemma}[theorem]{Lemma}
\newtheorem{corollary}[theorem]{Corollary}
\newtheorem{proposition}[theorem]{Proposition}
\newtheorem*{claim}{Claim}
\newtheorem*{theorem*}{Theorem}
\newtheorem*{corollary*}{Corollary}

\theoremstyle{definition}
\newtheorem{definition}[theorem]{Definition}
\theoremstyle{remark}
\newtheorem{remark}[theorem]{Remark}

\title{Characterizations of ordinal analysis}
\author{James Walsh}
\thanks{This paper is a synthesis of two earlier preprints: \cite{walsh2021characterization} and \cite{walsh2022robust2}. Thanks to Leszek Kolodziejczyk for catching a number of errors in an early version of this work and to Fedor Pakhomov for simplifying some of our proofs. Thanks also to the referee for many helpful comments and corrections and for suggesting the inclusion of \textsection \ref{conceptual}. Finally, thanks to Dan Appel, Antonio Montalb\'{a}n, and Benny Siskind for helpful discussion.}

\address{Sage School of Philosophy, Cornell University}
\email{jameswalsh@cornell.edu}

\begin{document}

\maketitle

\begin{abstract}
Ordinal analysis is a research program wherein recursive ordinals are assigned to axiomatic theories. According to conventional wisdom, ordinal analysis measures the strength of theories. Yet what is the attendant notion of strength? In this paper we present abstract characterizations of ordinal analysis that address this question. 

First, we characterize ordinal analysis as a partition of $\Sigma^1_1$-definable and $\Pi^1_1$-sound theories, namely, the partition whereby two theories are equivalent if they have the same proof-theoretic ordinal. We show that no equivalence relation $\equiv$ is finer than the ordinal analysis partition if both: (1) $T\equiv U$ whenever $T$ and $U$ prove the same $\Pi^1_1$ sentences; (2) $T\equiv T+U$ for every set $U$ of true $\Sigma^1_1$ sentences. In fact, no such equivalence relation makes a single distinction that the ordinal analysis partition does not make.

Second, we characterize ordinal analysis as an ordering on arithmetically-definable and $\Pi^1_1$-sound theories, namely, the ordering wherein $T\leq  U$ if the proof-theoretic ordinal of $T$ is less than or equal to the proof-theoretic ordinal of $U$. The standard ways of measuring the strength of theories are consistency strength and inclusion of $\Pi^0_1$ theorems. We introduce analogues of these notions---$\Pi^1_1$-reflection strength and inclusion of $\Pi^1_1$ theorems---in the presence of an oracle for $\Sigma^1_1$ truths, and prove that they coincide with the ordering induced by ordinal analysis.

\end{abstract}

\section{Introduction}

Measuring the strength of axiomatic theories is a recurring motif in mathematical logic. Ordinal analysis is one research program within mathematical logic wherein this motif emerges. In ordinal analysis, an axiomatic theory is associated, in a principled way, with a recursive ordinal called its \emph{proof-theoretic ordinal}. It is often claimed that by calculating the proof-theoretic ordinal of a theory, we thereby measure its strength. Yet what exactly is the attendant notion of strength? That is, if we have determined the proof-theoretic ordinal of a theory, in what sense have we determined its strength?

To answer this question, we will give two characterizations of ordinal analysis that make no reference to the notion of ``proof-theoretic ordinals.'' Our first theorem characterizes the \emph{partition} of theories induced by ordinal analysis and our second theorem characterizes the \emph{ordering} on theories induced by ordinal analysis. These characterizations arise from different directions, so it makes the most sense to describe them separately.

\subsection{The Ordinal Analysis Partition}

One way of interpreting ordinal analysis is as a classification program. That is, ordinal analysis induces a partition of theories, where $T$ and $U$ are equivalent if $T$ and $U$ have the same proof-theoretic ordinal. Given this perspective, to understand the sense in which ordinal analysis measures strength, we should know what features of theories this partition is sensitive to.

To this end, we characterize the ordinal analysis partition in abstract terms. In particular, we characterize it as the most fine-grained partition satisfying two natural conditions; these conditions are articulated without using the notion ``proof-theoretic ordinal.'' The first main theorem is that no partition satisfying these conditions makes a distinction that the ordinal analysis partition does not make. This characterization is evidence of the naturalness and robustness of the ordinal analysis partition.

The theories that we will work with are $\Sigma^1_1$-definable and $\Pi^1_1$-sound extensions of $\mathsf{ACA}_0$. For more information on $\mathsf{ACA}_0$, see \cite{simpson2009subsystems}. For other results on $\Sigma^1_1$-definable and $\Pi^1_1$-sound theories, see \cite{walsh2021incompleteness}.

Before continuing, let's recall a standard definition of the proof-theoretic ordinal of a theory. In what follows, $\mathsf{WF}(\prec)$ is a sentence expressing the well-foundedness of $\prec$:
$$\mathsf{WF}(\prec):= \forall X\big(\exists x\in X \to \exists x\in X \; \forall y \in X \; y \not\prec x \big).$$
\begin{definition}
$|T|_{\mathsf{WF}}$ is the supremum of the order types of the primitive recursive presentations $\prec$ of well-orderings such that $T\vdash \mathsf{WF}(\prec)$.
\end{definition}

Now let's introduce the two conditions that we will use to characterize ordinal analysis. Let $T\equiv_{\Pi^1_1}U$ mean that $T$ and $U$ have the same $\Pi^1_1$ theorems. The first important property of ordinal analysis is that $|T|_{\mathsf{WF}}=|U|_{\mathsf{WF}}$ whenever $T\equiv_{\Pi^1_1}U$. This is clearly true since well-foundedness claims for primitive recursive well-orderings are $\Pi^1_1$. To introduce the second condition, we recall a theorem that is commonly attributed to Kreisel:\footnote{For proofs of the original Kreisel theorem, see \cite{pohlers2008proof} Theorem 6.7.5 or \cite{rathjen1999realm} Proposition 2.24. }\footnote{Note that $T+\varphi$ is $T\cup\{\varphi\}$ and $T+V$ is $T\cup V$.}
\begin{theorem}[Kreisel]\label{kreisel}
For any recursively axiomatized $\Pi^1_1$-sound extension $T$ of $\mathsf{ACA}_0$ and true $\Sigma^1_1$ sentence $\varphi$, $|T|_{\mathsf{WF}}=|T+\varphi|_{\mathsf{WF}}$.
\end{theorem}
A version of Kreisel's theorem also holds for $\Sigma^1_1$-definable $\Pi^1_1$-sound extensions of $\mathsf{ACA}_0$. Thus, the following claims hold for all $\Sigma^1_1$-definable and $\Pi^1_1$-sound $T$ and $U$ extending $\mathsf{ACA}_0$:
\begin{enumerate}
    \item If $T\equiv_{\Pi^1_1}U$, then $|T|_{\mathsf{WF}}=|U|_{\mathsf{WF}}$.
    \item If $V$ consists of true $\Sigma^1_1$ sentences, then $|T|_{\mathsf{WF}}=|T+V|_{\mathsf{WF}}$.
\end{enumerate}

Let's say that an equivalence relation $\equiv$ on theories is \emph{good} if it has both those properties.

\begin{definition}\label{good}
An equivalence relation $\equiv$ is \emph{good} if for all $\Sigma^1_1$-definable and $\Pi^1_1$-sound $T$ and $U$ extending $\mathsf{ACA}_0$:
\begin{enumerate}
    \item If $T\equiv_{\Pi^1_1}U$, then $T\equiv U$.
    \item If $V$ consists of true $\Sigma^1_1$ sentences, then $T\equiv T+V$.
\end{enumerate}
\end{definition}

Our first theorem is that no good partition makes a single distinction that the ordinal analysis partition does not make.
\begin{theorem}\label{main-thm}
Let $\equiv$ be good. Let $T$ and $U$ be $\Sigma^1_1$-definable and $\Pi^1_1$-sound extensions of $\mathsf{ACA}_0$ such that $|T|_{\mathsf{WF}}=|U|_{\mathsf{WF}}$. Then $T\equiv U$.
\end{theorem}
This theorem yields an abstract characterization of ordinal analysis. Indeed, the ordinal analysis partition is the finest good partition.


\subsection{The Ordinal Analysis Ordering}
Ordinal analysis also induces an ordering on theories where $T\leq U$ if $|T|_{\mathsf{WF}}\leq|U|_{\mathsf{WF}}$. We will also characterize this ordering in abstract terms. To motivate this characterization, let's briefly put ordinal analysis to the side and focus on the question: What is the structure of axiomatic theories ordered by relative logical strength? To answer this question we must clarify what we mean by ``logical strength.'' There are many notions of logical strength in the literature, but these notions typically coincide when restricted to ``naturally occurring'' axiomatic theories. Nevertheless, it is possible to concoct axiomatic theories in an \emph{ad hoc} fashion so that these notions of logical strength come apart. That is, the different notions of logical strength coincide on natural theories but come apart as means of comparing axiomatic theories in general. Perhaps the most well-known notion for comparing the logical strength of theories is consistency strength over a suitable base theory:


\begin{definition}
$T\leq_{\mathsf{Con}} U \defiff \mathsf{ACA}_0 \vdash \mathsf{Con}(U) \to \mathsf{Con}(T).$
\end{definition}

It is convenient for us to use $\mathsf{ACA}_0$ as our base theory. One sometimes sees different choices for the base theory; common choices include $\mathsf{EA}$, $\mathsf{PRA}$, and $\mathsf{PA}$. Note that since $\mathsf{ACA}_0$ is conservative over $\mathsf{PA}$, consistency strength over $\mathsf{ACA}_0$ and consistency strength over $\mathsf{PA}$ are actually the same. 

Another common way to compare the strength of theories is to compare their $\Pi^0_1$ consequences. 

\begin{definition}
$T\subseteq_{\Pi^0_1} U \defiff$  for every $\varphi\in\Pi^0_1$, if $T\vdash \varphi$ then $U\vdash \varphi$.
\end{definition}

It is often claimed that, when we restrict our attention to ``natural'' theories, the $\subseteq_{\Pi^0_1}$ ordering coincides with relative consistency strength \cite{steel2014godel}. However, these notions do not coincide in general.\footnote{For a counter-example, consider any consistent $T$ and let $R_T$ be the Rosser sentence for $T$. Then $T+R_T\not\subseteq_{\Pi^0_1} T$ but $\mathsf{ACA}_0\vdash \mathsf{Con}(T) \to \mathsf{Con}(T+R_T)$, i.e., $T+R_T \leq_{\mathsf{Con}}T$.}

The orderings $\leq_{\mathsf{Con}}$ and $\subseteq_{\Pi^0_1}$ are neither pre-linear nor pre-well-founded. That is, in both orderings, there are incomparable elements and infinite descending sequences. Remarkably, both of these features disappear when we restrict our attention to the natural theories. The restriction of $\leq_{\mathsf{Con}}$ to natural theories coincides with the restriction of $\subseteq_{\Pi^0_1}$ to natural theories, and these restrictions engender a pre-well-ordering.\footnote{At least, this seems to be the majority opinion; see, for instance, discussions in \cite{montalban2019martin, shelah2003logical, steel2014godel}. There have been dissenting voices, however; see  \cite{hamkins2022nonlinearity, hauser2014strong}.}

If it is true that natural axiomatic theories are pre-well-ordered by logical strength---that is, if it is not merely an illusion engendered by a paucity of examples---then one might like to prove that it is true. However, without a precise mathematical definition of the “natural'' axiomatic theories, it is not clear how to prove this claim. It is not even clear how to state it mathematically.

As part of our characterization of the ordinal analysis ordering, we will introduce analogues of the aforementioned consistency strength and $\Pi^0_1$ theorem inclusion orderings that are \emph{actually} pre-well-ordered. To do this, we will tweak these orderings in two ways.

First, we will replace the notion of \emph{provability} with the notion of \emph{provability in the presence of an oracle for $\Sigma^1_1$ truths.} A theory $T$ will prove a sentence $\varphi$ in the presence of such an oracle if $T+\psi$ proves $\varphi$ for some true $\Sigma^1_1$ $\psi$. We introduce the following notation to capture this idea:

\begin{definition}
For a complexity class $\Gamma$, we define $T\vdash^{\Sigma^1_1}\varphi \defiff $ there is a true $\psi\in\Gamma$ such that $T+\psi\vdash\varphi$.
\end{definition}

Second, we focus our attention on the $\Pi^1_1$ consequences of theories rather than the $\Pi^0_1$ consequences of theories. This shift in perspective yields the following analogue of $\subseteq_{\Pi^0_1}$:

\begin{definition}
$T\subseteq^{\Sigma^1_1}_{\Pi^1_1} U \defiff \text{ for all $\varphi \in \Pi^1_1$, if $T\vdash ^{\Sigma^1_1} \varphi$ then  $U\vdash^{\Sigma^1_1}  \varphi$.}$
\end{definition}
The only theories we will consider are extensions of $\mathsf{ACA}_0$; hence $T\vdash\varphi$ is equivalent to $T\vdash^{\Sigma^0_1}\varphi$.\footnote{This is because $\mathsf{ACA}_0$ proves every true $\Sigma^0_1$ sentence. Indeed, comparably weak subsystems of $\mathsf{ACA}_0$ are also $\Sigma^0_1$-complete.} Thus, $T\subseteq_{\Pi^0_1}U$ is equivalent to $T\subseteq^{\Sigma^0_1}_{\Pi^0_1}U$. Note that $\subseteq^{\Sigma^1_1}_{\Pi^1_1}$ is just the result of changing $\subseteq^{\Sigma^0_1}_{\Pi^0_1}$ by replacing $\Sigma^0_1$  with $\Sigma^1_1$ and $\Pi^0_1$ with $\Pi^1_1$.

Before stating the analogue of $\leq_{\mathsf{Con}}$, let's note that a theory is consistent just in case all of its $\Pi^0_1$ consequences are true. This means that $T\leq_{\mathsf{Con}}U$ is equivalent to the following claim:
$$\mathsf{ACA}_0 \vdash^{\Sigma^0_1} \mathsf{RFN}_{\Pi^0_1}(U) \to \mathsf{RFN}_{\Pi^0_1}(T),$$
where $\mathsf{RFN}_{\Pi^0_1}(T)$ is a formula expressing that all of $T$'s $\Pi^0_1$ consequences are true.

We will be interested in $\Pi^1_1$-soundness, where a theory is $\Pi^1_1$-sound just in case all its $\Pi^1_1$ consequences are true. We can formalize the $\Pi^1_1$-soundness of $T$ with a single sentence in $\mathsf{ACA}_0$:
$$\mathsf{RFN}_{\Pi^1_1}(T) : =\forall \varphi \in \Pi^1_1 \big( \mathsf{Pr}_T(\varphi) \to \mathsf{True}_{\Pi^1_1}(\varphi)  \big).$$
Note that $\mathsf{Pr}_T$ here picks out ordinary provability from $T$, not provability in the presence of an oracle. Hence, we have the following analogue of $\leq_{\mathsf{Con}}$:

\begin{definition}
$T\leq^{\Sigma^1_1}_{\mathsf{RFN_{\Pi^1_1}}} U \defiff \mathsf{ACA}_0 \vdash^{\Sigma^1_1} \mathsf{RFN}_{\Pi^1_1}(U) \to \mathsf{RFN}_{\Pi^1_1}(T).$
\end{definition}

Recall that, according to conventional wisdom, calculating the proof-theoretic ordinal of a theory is a means of measuring its logical strength. However, note that the ordering of theories induced by ordinal analysis:
$$T\leq_{\mathsf{WF} }U \defiff |T|_{\mathsf{WF}}\leq |U|_{\mathsf{WF}}$$
is a pre-well-ordering since the ordinals are well-ordered. Hence, $\leq_{\mathsf{WF} }$ cannot strictly coincide with either $\leq_{\mathsf{Con}}$ or $\subseteq_{\Pi^0_1}$.

Nevertheless, in the presence of an oracle for $\Sigma^1_1$ truths, we can vindicate the common wisdom that ordinal analysis is a means of measuring the logical strength of theories. Indeed, our second main theorem in this paper is the following:
\begin{theorem}\label{main-intro}
For $\Pi^1_1$-sound arithmetically definable $T$ and $U$ extending $\mathsf{ACA}_0:$
$$T\subseteq_{\Pi^1_1}^{\Sigma^1_1}U \Longleftrightarrow T\leq_{\mathsf{WF}} U  \Longleftrightarrow T\leq^{\Sigma^1_1}_{\mathsf{RFN_{\Pi^1_1}}} U.$$
\end{theorem}

Since the ordinals are well-ordered, this immediately yields the following corollary:
\begin{corollary}\label{cor-intro}
The relations $\subseteq_{\Pi^1_1}^{\Sigma^1_1}$ and $\leq^{\Sigma^1_1}_{\mathsf{RFN}_{\Pi^1_1}}$ pre-well-order the $\Pi^1_1$-sound arithmetically definable extensions of $\mathsf{ACA}_0$.
\end{corollary}
Note that in the statement of Theorem \ref{main-intro} and its corollary, we have dropped the non-mathematical quantification over ``natural'' theories.



\subsection{The Conceptual Framework}\label{conceptual}

In this paper we will be dealing with $| \; |_{\mathsf{WF}}$ at a rather abstract level. In particular, we will be concerned \emph{only} with the partition and ordering that $| \; |_{\mathsf{WF}}$ induces on theories; we will not be concerned with the various other projects that usually attend the calculation of $| \; |_{\mathsf{WF}}$. This may seem like a myopic or naive perspective on ordinal analysis. For instance, Rathjen writes:
\begin{displayquote}[Rathjen \cite{rathjen1999realm}, p. 220]
In the literature, the result of an ordinal analysis of a given theory $T$ is often stated in a rather terse way by saying that the supremum of the provable recursive well-orderings \dots is a certain ordinal $\alpha$. This is at best a shorthand for a much more informative statement. From questions that I've been asked over the years, I know that sloppy talk about proof-theoretic ordinals has led to misconceptions about ordinal-theoretic proof theory.
\end{displayquote}
He also writes that  ``in general, the mere knowledge of'' $|T|_{\mathsf{WF}}$ ``is not the goal of an ordinal analysis of $T$'' (\cite{rathjen1999realm}, p. 237). Instead, Rathjen emphasizes that the calculation of $|T|_{\mathsf{WF}}$ usually yields a characterization of $T$'s provably recursive functions, examples of $T$-independent combinatorial principles, and proof-theoretic reductions between axiom systems. One might worry that my purported ``characterizations of ordinal analysis'' are attending only to the ``shorthand'' and missing out on the ``much more informative'' aspects of ordinal analysis.

Before addressing this concern it is worth flagging that there is no consensus concerning the \emph{primary} benefit of ordinal analysis. The original motivation for ordinal analysis was to develop quasi-finitary consistency proofs in the spirit of Hilbert's Program.\footnote{Gentzen \cite{gentzen1969consistency} described his consistency proof as a ``real vindication of the disputable parts of elementary number theory.'' Likewise, Takeuti \cite{takeuti1975consistency} called
Gentzen’s proof ``greatly reassuring'' and wrote that it enhanced his ``confidence in the consistency and truth of Peano arithmetic''.} Nevertheless, the value of these consistency proofs has been widely criticized. According to Kreisel \cite{kreisel1979formal}, Tarski said that Gentzen's proof only increased his confidence in $\mathsf{PA}$'s consistency ``by an epsilon.''

Even the value of ordinal analysis for securing its standard proof-theoretic corollaries has been challenged. For instance, Kentaro Sato recently gave ordinal-free proofs of the reductions from $\mathbf{\Sigma^1_2}\text{-}\mathsf{AC}+\mathsf{BI}$ to $\mathbf{T}_0$ \cite{sato2015new} and $\mathbf{\Sigma^1_1}\text{-}\mathsf{DC}_0+(\Pi^1_{n+1}\text{-}\mathsf{Ind})$ to $\mathbf{\Delta^1_1}\text{-}\mathsf{CA}_0+(\Pi^1_{n+1}\text{-}\mathsf{Ind})$ \cite{sato2022new}. He wrote that this leaves no known reducibility result between classical theories whose only proof uses ordinal analysis (\cite{sato2022new}, \textsection 1.2). Sato's proofs use ``relatively easy proof-theoretic techniques'' rather than the heavy machinery used to calculate proof-theoretic ordinals. Moreover, the easy techniques are adaptable to subsystems of second-order arithmetic and set theory that are currently beyond the reach of ordinal analysis.

Perhaps ordinal analysis has no \emph{primary} benefit. Yet, whatever the status of the particular applications just reviewed, ordinal analysis remains interesting. Indeed, in the author's opinion, one of the most fascinating aspects of ordinal analysis \emph{as such} is that (1) it pre-well-orders axiomatic theories and (2) this ordering is clearly connected to the typical orderings of logical strength, e.g. consistency strength.\footnote{This is not to say that it is clear what the connection is.} Recall that one of the central questions in the foundations of mathematics is: Why are the natural axiomatic theories pre-well-ordered by consistency strength?\footnote{See the discussion in the previous subsection.} (1) and (2) suggest that ordinal analysis may be relevant to answering this question.

Even those who emphasize other aspects of ordinal analyses promote the intuitive picture of ordinal analysis as a means of ranking axiom systems according to their ``strength.'' Rathjen writes:
\begin{displayquote}[Rathjen \cite{rathjen1999realm}, p. 219]
A central theme running through all the main areas of Mathematical Logic is the classification of sets, functions or theories, by means of transfinite hierarchies whose ordinal levels measure their ‘rank’ or ‘complexity’ in some sense appropriate to the underlying context. In Proof Theory this is manifest in the assignment of ‘proof theoretic ordinals’ to theories, gauging their ‘consistency strength’ and `computational power'.
\end{displayquote}

Of course, as is well known, $|T|_{\mathsf{WF}}$ does not \emph{exactly} gauge the consistency strength of theories; some theories that are not equi-consistent share their proof-theoretic ordinal. Similar issues attend computational power. It is worth noting that variants of $|T|_{\mathsf{WF}}$ have been introduced that are designed to gauge consistency strength and computational power; for instance, see $|T|_{\Pi^0_1}$ and $|T|_{\Pi^0_2}$ in \cite{beklemishev2005reflection}. However, these values are notation dependent. Relative to particularly ``natural'' choices of ordinal notation systems, these values coincide with $|T|_{\mathsf{WF}}$, at least for many choices of $T$. This suggests an intimate connection between $|T|_{\mathsf{WF}}$, consistency strength and computational power, but it is difficult to prove anything to that effect given the notation dependence of $|T|_{\Pi^0_1}$ and $|T|_{\Pi^0_2}$.

So $|T|_{\mathsf{WF}}$ does not exactly gauge consistency strength and computational power. Then what \emph{does} $|T|_{\mathsf{WF}}$ gauge? The results in this paper answer that question; in particular, we characterize in exact terms the analogue of consistency strength that ordinal analysis is actually measuring. The fact that this analogue of consistency strength induces a pre-well-ordering on axiom systems is, in the author's opinion, interesting for reasons independent of ordinal analysis.

To prevent a misunderstanding, let's note that, for any theory $T$, our characterizations do not attach any obvious significance to the \emph{ordinal number} $|T|_{\mathsf{WF}}$. Rather, they attach significance to the position of $T$ in the ordinal analysis partition and ordering. To make this point a bit more explicit, define $|T|^\star_{\mathsf{WF}}:=|T|_{\mathsf{WF}}+1$. Then $|T|^\star_{\mathsf{WF}}= |U|^\star_{\mathsf{WF}}$ if and only if $|T|_{\mathsf{WF}}= |U|_{\mathsf{WF}}$. Likewise, $|T|^\star_{\mathsf{WF}}\leq |U|^\star_{\mathsf{WF}}$ if and only if $|T|_{\mathsf{WF}}\leq |U|_{\mathsf{WF}}$. So our characterizations of the ordinal analysis partition and ordering do not depend on any assumption that $|T|_{\mathsf{WF}}$ rather than $|T|^\star_{\mathsf{WF}}$ is the ``correct'' ordinal value of $T$. Rather, the reason for our interest in $|T|_{\mathsf{WF}}$ is that it exhibits the relative placement of $T$ in the ordinal analysis ordering and partition. So Rathjen's remark that ``the mere knowledge'' of $|T|_{\mathsf{WF}}$ is not a sensible goal for an ordinal analysis of $T$ still stands.\footnote{It is worth mentioning that Pohlers has proved many results at our same level of abstraction, i.e., he has proved results about $| \; |_{\mathsf{WF}}$ \emph{as such}. However, his particular proposals about the significance of this ordering are not entirely the same as the ones we will pursue here; for instance, see his notion of the $\Pi^1_1$-spectrum of a theory in \cite{pohlers1998subsystems}.}


\subsection{Outline of the Paper}

Our main goals in this paper are conceptual rather than technical. Though our arguments are elementary, the characterizations they engender seem to have been heretofore unnoticed.

The main technical component of our characterizations of ordinal analysis is the following fact, which states that well-foundedness is a universal $\Pi^1_1$ property, provably in $\mathsf{ACA}_0$.
\begin{theorem}\label{kleene-brouwer}
For every $\Pi^1_1$ sentence $\varphi$ there is a primitive recursive presentation $\prec$ of a linear ordering such that $\mathsf{ACA}_0 \vdash \varphi \leftrightarrow \mathsf{WF}(\prec)$.
\end{theorem}
The proof of this result essentially involves the construction of Kleene--Brouwer orderings in $\mathsf{ACA}_0$. For details see \cite{simpson2009subsystems} Lemmas V.1.4 and V.1.8.

The characterization of ordinal analysis as a partition was inspired by Montalb\'{a}n's \cite{montalban2017degree} characterization of the partition that identifies reals $A$ and $B$ when $\omega_1^A=\omega_1^B$. Montalb\'{a}n's result was suggestive because there are other analogies between the equivalence relation $\omega_1^A=\omega_1^B$ and the equivalence relation $|T|_{\mathsf{WF}}=|U|_{\mathsf{WF}}$; see the author's work with Lutz \cite{lutz2020incompleteness} for details. The characterization of ordinal analysis as an ordering extends earlier work by Pakhomov and the author \cite{pakhomov2021reflection}. They show that, in a large swathe of cases, proof-theoretic ordinals coincide with ranks of theories in a proof-theoretic reflection ordering; this latter ordering is not linear, however, so this earlier work does not yield Theorem \ref{main-intro}.

Here is our plan for the rest of the paper. In \textsection \ref{kreisel-section}, we will characterize the partition induced by ordinal analysis. In \textsection \ref{main-section} we will characterize the ordering induced by ordinal analysis. We will also prove a negative theorem to the effect that Theorem \ref{main-intro} cannot be strengthened.


\section{Ordinal Analysis as a Partition}\label{kreisel-section}

Before diving into the proof of Theorem \ref{main-thm}, we should check that the ordinal analysis partition is a good partition. Thus, we will first derive a version of Kreisel's Theorem \ref{kreisel} for theories that are $\Sigma^1_1$-definable. For proofs of the original Kreisel theorem, see \cite{pohlers2008proof} Theorem 6.7.5 or \cite{rathjen1999realm} Proposition 2.24. 

\begin{proposition}\label{prop}
Let $T$ be a $\Pi^1_1$-sound extension of $\mathsf{ACA}_0$. Then $|T|_{\mathsf{WF}}=|T+V|_{\mathsf{WF}}$ for any set $V$ of true $\Sigma^1_1$ sentences.
\end{proposition}

\begin{proof}
Clearly $|T|_{\mathsf{WF}}\leq |T+V|_{\mathsf{WF}}$. It remains to show that  $|T+V|_{\mathsf{WF}}\leq |T|_{\mathsf{WF}}$.

Suppose that $T+V\vdash\mathsf{WF}(\alpha)$ for primitive recursive $\alpha$; it suffices to show that $T\vdash\mathsf{WF}(\beta)$ for some primitive recursive $\beta$ such that $\mathsf{otyp}(\beta)\geq \mathsf{otyp}(\alpha)$. Note that only finitely many sentences from $T$ and $V$ are used in the proof exhibiting that $T+V\vdash \mathsf{WF}(\alpha)$. So we have:
$$\mathsf{ACA}_0+\tau_1+\dots+\tau_n+\nu_1+\dots+\nu_k\vdash\mathsf{WF}(\alpha)$$
where $\tau_1,\dots,\tau_n$ are axioms of $T$ and $\nu_1,\dots,\nu_k$ are from $V$. Note that $$(\nu_1\wedge\dots\wedge\nu_k)$$ is a true $\Sigma^1_1$ sentence. Moreover, note that
$$\mathsf{ACA}_0+\tau_1+\dots+\tau_n$$ is finitely axiomatized and $\Pi^1_1$-sound. Thus, Kreisel's original Theorem \ref{kreisel} applies to the theory $$\mathsf{ACA}_0+\tau_1+\dots+\tau_n.$$ So we infer that $$|\mathsf{ACA}_0+\tau_1+\dots+\tau_n|_{\mathsf{WF}}= |\mathsf{ACA}_0+\tau_1+\dots+\tau_n+\nu_1+\dots+\nu_k|_{\mathsf{WF}}.$$
It follows that $T\vdash\mathsf{WF}(\beta)$ for some primitive recursive $\beta$ such that $\mathsf{otyp}(\beta)\geq \mathsf{otyp}(\alpha)$. \end{proof}


\begin{remark}
Note that Proposition \ref{prop} is provable in $\mathsf{ACA}_0$. Indeed, proofs of the original Kreisel theorem (e.g., the proof of Proposition 2.2.4 in \cite{rathjen1999realm}) are valid in $\mathsf{ACA}_0$. 
\end{remark}

Before turning to our first main theorem, let's record one small lemma.

\begin{lemma}\label{lemma}
If $T$ is $\Pi^1_1$ sound and $\varphi$ is true $\Sigma^1_1$, then $T+\varphi$ is $\Pi^1_1$ sound.
\end{lemma}

\begin{proof}
Let $T+\varphi\vdash \psi$ for $\psi\in\Pi^1_1$. Then $T\vdash \varphi\to\psi$. Since $T$ is $\Pi^1_1$ sound, $\varphi\to\psi$ is true. Since $\varphi$ is true, $\psi$ must be true.
\end{proof}

Now we turn to our first main theorem. Our goal is to characterize the ordinal analysis partition in terms of the \emph{good partitions}; for the definition of good partitions see Definition \ref{good}. It is immediate from its definition that the ordinal analysis partition has the first property. In the beginning of this section we saw that the second claim is true of the ordinal analysis partition. We will now show that they are \emph{not} both true of any partition that makes distinctions not made by the ordinal analysis partition.

We restate Theorem \ref{main-thm} here for convenience.

\begin{theorem*}
Let $\equiv$ be good. Let $T$ and $U$ be $\Sigma^1_1$-definable and $\Pi^1_1$-sound extensions of $\mathsf{ACA}_0$ such that $|T|_{\mathsf{WF}}=|U|_{\mathsf{WF}}$. Then $T\equiv U$.
\end{theorem*}

\begin{proof}
Suppose that $|T|_{\mathsf{WF}}=|U|_{\mathsf{WF}}$.

Let $T_{\Pi^1_1}$ be the set of $\Pi^1_1$ theorems of $T$. Let $\varphi\in T_{\Pi^1_1}$. By Theorem \ref{kleene-brouwer}, $\mathsf{ACA}_0 \vdash \varphi\leftrightarrow \mathsf{WF}(\alpha)$ for some primitive recursive $\alpha$. So $T\vdash \mathsf{WF}(\alpha)$. So $\alpha<|T|_{\mathsf{WF}}=|U|_{\mathsf{WF}}$. So there is some primitive recursive $\beta$ such that $\mathsf{otyp}(\alpha)\leq \mathsf{otyp}(\beta)$ and $U\vdash \mathsf{WF}(\beta)$. Now the sentence $\exists f \mathsf{Emb}(f,\alpha,\beta)$, which formalizes the claim that $\alpha$ embeds into $\beta$, is true $\Sigma^1_1$. Note that $U+\exists f \mathsf{Emb}(f,\alpha,\beta)  \vdash \mathsf{WF}(\alpha)$, so $U+\exists f \mathsf{Emb}(f,\alpha,\beta) \vdash \varphi$.

Likewise, let $U_{\Pi^1_1}$ be the set of $\Pi^1_1$ theorems of $U$. As above, for each $\psi\in U_{\Pi^1_1}$, we may find true $\Sigma^1_1$ sentences of the form $\exists f \mathsf{Emb}(f,\gamma,\delta)$ so that:
$$T + \exists f \mathsf{Emb}(f,\gamma,\delta) \vdash \psi.$$

We will enrich $U$ with all the sentences $\exists f \mathsf{Emb}(f,\alpha,\beta)$ and $\exists f \mathsf{Emb}(f,\gamma,\delta)$ that can be found as in the previous paragraphs. That is, we define the theory $\widehat{U}$ as follows: $\theta$ belongs to $\widehat{U}$ if and only if one of the following holds:
\begin{enumerate}
    \item $\theta$ belongs to $U$;
    \item $\theta$ has the form $\exists f \mathsf{Emb}(f,\alpha,\beta)$ where $\alpha$ and $\beta$ are primitive recursive and
$$ \exists \varphi\in T_{\Pi^1_1} \Big( \mathsf{ACA}_0\vdash \varphi \leftrightarrow\mathsf{WF}(\alpha) \text{ and }  U\vdash \mathsf{WF}(\beta) \text{ and } \mathsf{otyp}(\alpha)\leq\mathsf{otyp}(\beta) \Big);  $$
\item $\theta$ has the form $\exists f \mathsf{Emb}(f,\gamma,\delta)$ where $\gamma$ and $\delta$ are primitive recursive and
$$ \exists \psi\in U_{\Pi^1_1} \Big( \mathsf{ACA}_0\vdash \psi \leftrightarrow\mathsf{WF}(\gamma) \text{ and }  T\vdash \mathsf{WF}(\delta) \text{ and } \mathsf{otyp}(\gamma)\leq\mathsf{otyp}(\delta) \Big).  $$
\end{enumerate} 

We define the theory $\widehat{T}$ in the exact same manner except that we replace clause (1) above with the condition ``$\theta$ belongs to $T$.''

Note that it is immediate from the construction of $\widehat{T}$ and $\widehat{U}$---in particular, from the way the sentences $\exists f \mathsf{Emb}(f,\alpha,\beta)$ and $\exists f \mathsf{Emb}(f,\gamma,\delta)$ were selected---that both prove all of the $\Pi^1_1$ theorems of $T$ and all the $\Pi^1_1$ theorems of $U$.

\begin{claim}
$T\equiv \widehat{T}$ and $U\equiv \widehat{U}$.
\end{claim}

$\widehat{U}$ is an extension of $U$ by true $\Sigma^1_1$ sentences. This implies that $U\equiv \widehat{U}$, since $\equiv$ is good. Likewise, $T\equiv \widehat{T}$. 

\begin{claim}
$\widehat{T} \equiv \widehat{U}$.
\end{claim}

First we note that $\widehat{T} \equiv_{\Pi^1_1} \widehat{U}$. To see that $\widehat{T}\supseteq_{\Pi^1_1}\widehat{U}$, suppose $\widehat{U}\vdash \theta$ where $\theta$ is $\Pi^1_1$. Then $U+\sigma\vdash \theta$ where $\sigma$ is a conjunction of $\Sigma^1_1$ claims that were added to $U$ to get $\widehat{U}$. So $U\vdash \sigma\to\theta$. Note that $\sigma\to\theta$ is $\Pi^1_1$. But then $\widehat{T}\vdash \sigma\to\theta$, since we constructed $\widehat{T}$ so that it would prove all $\Pi^1_1$ theorems of $U$. Note that $\sigma$ is also a conjunction of $\Sigma^1_1$ claims that were added to $T$ to get $\widehat{T}$, whence $\widehat{T}\vdash \theta$. A symmetric argument shows that $\widehat{T}\subseteq_{\Pi^1_1}\widehat{U}$.

Since $\equiv$ is good, the claim follows as long as $\widehat{T}$ and $\widehat{U}$ are $\Sigma^1_1$-definable and $\Pi^1_1$-sound. 

To see that they are $\Pi^1_1$-sound: Suppose that $\widehat{T}\vdash \theta$, where $\theta$ is $\Pi^1_1$. Then $T+\sigma\vdash \theta$, where $\sigma$ is a conjunction of true $\Sigma^1_1$ claims. But $T+\sigma$ is $\Pi^1_1$-sound by Lemma \ref{lemma}, whence $\theta$ is true. Of course, a symmetric argument applies to $\widehat{U}$.

To see that they are $\Sigma^1_1$-definable: A sentence $\theta$ belongs to $\widehat{U}$ if and only if it satisfies any of clauses (1)--(3) above. Clause (1) is a $\Sigma^1_1$ condition since $U$ is $\Sigma^1_1$-definable. Clauses (2) and (3) are similar to each other. Let's look only at clause (2).

Having the syntactic form $\exists f \mathsf{Emb}(f,\alpha,\beta)$ for primitive recursive $\alpha$ and $\beta$ is arithmetic. $T_{\Pi^1_1}$ is a $\Sigma^1_1$-definable set, since $T$-provability is $\Sigma^1_1$. The first conjunct within the parentheses is $\Sigma^0_1$. The second conjunct is $\Sigma^1_1$ since $U$ is. The third conjunct $\mathsf{otyp}(\alpha)\leq\mathsf{otyp}(\beta)$ is formalized by the claim $\exists f \mathsf{Emb}(f,\alpha,\beta)$, which is also $\Sigma^1_1$.

The same observations show that $\widehat{T}$ is $\Sigma^1_1$-definable. This concludes the proof of the claim.

It immediately follows from the two claims  that $T\equiv \widehat{T}\equiv \widehat{U}\equiv U$, whence $T\equiv U$.
\end{proof}


\section{Ordinal Analysis as an Ordering}\label{main-section}

In this section we will characterize the ordering on theories induced by ordinal analysis. Recall that our second main theorem is that for $\Pi^1_1$-sound arithmetically definable $T$ and $U$ extending $\mathsf{ACA}_0:$
$$T\subseteq_{\Pi^1_1}^{\Sigma^1_1}U \Longleftrightarrow T\leq_{\mathsf{WF}} U  \Longleftrightarrow T\leq^{\Sigma^1_1}_{\mathsf{RFN_{\Pi^1_1}}} U.$$

In the proof we will once again make use of Theorem \ref{kleene-brouwer}, which says that well-foundedness is a universal $\Pi^1_1$ property. We also have a uniform version of Theorem \ref{kleene-brouwer}; see the proof of \cite{simpson2009subsystems} Lemma V.1.8 but appeal to Theorem V.1.7$'$ rather than Theorem V.1.7.

\begin{theorem}\label{kb-uniform}
For any $\Pi^1_1$ formula $\varphi(x)$, there is a primitive recursive family $\langle \beta_x \mid x\in\mathbb{N}\rangle$ of primitive recursive linear orders such that $\mathsf{ACA}_0\vdash \forall x\big( \varphi(x) \leftrightarrow \mathsf{WF}(\beta_x)\big)$.
\end{theorem}








Before continuing, let's record a small lemma that we will use repeatedly.\footnote{This lemma is actually implicit in the proof of Theorem \ref{main-thm}. It did not make sense to isolate it there, though, because we needed to make use of the specific $\Sigma^1_1$ sentence used to witness $\vdash^{\Sigma^1_1}$.}

\begin{lemma}\label{embeddings}
If $\alpha$ is a primitive recursive well-ordering and $\mathsf{otyp}(\alpha)<|T|_{\mathsf{WF}}$, then $T\vdash^{\Sigma^1_1}\mathsf{WF}(\alpha)$.
\end{lemma}

\begin{proof}
Suppose that $\mathsf{otyp}(\alpha)<|T|_{\mathsf{WF}}$. Then there is some primitive recursive $\beta$ such that $\mathsf{otyp}(\alpha)\leq\mathsf{otyp}(\beta)$ and $T\vdash \mathsf{WF}(\beta)$. The statement $\exists f \mathsf{Emb}(f,\alpha,\beta)$, which says that $\alpha$ embeds into $\beta$, is true $\Sigma^1_1$. Since $T+\exists f \mathsf{Emb}(f,\alpha,\beta)\vdash \mathsf{WF}(\alpha)$, we infer that $T\vdash^{\Sigma^1_1}\mathsf{WF}(\alpha)$.
\end{proof}

\subsection{The First Equivalence}

The first of the two equivalences in our second main theorem has a straightforward proof.

\begin{lemma}\label{easy}
For all $\Sigma^1_1$-definable and $\Pi^1_1$-sound extensions $T$ and $U$ of $\mathsf{ACA}_0$:
$$T\subseteq_{\Pi^1_1}^{\Sigma^1_1} U \Longleftrightarrow |T|_{\mathsf{WF}} \leq |U|_{\mathsf{WF}}.$$
\end{lemma}

\begin{proof}
\emph{Left to right:} Assume that $T\subseteq_{\Pi^1_1}^{\Sigma^1_1} U$. Let $\mathsf{otyp}(\alpha) <|T|_{\mathsf{WF}}$ for some primitive recursive $\alpha$. By Lemma \ref{embeddings}, $T\vdash^{\Sigma^1_1}\mathsf{WF}(\alpha)$. Since $T\subseteq_{\Pi^1_1}^{\Sigma^1_1} U$, it follows that $U\vdash^{\Sigma^1_1}\mathsf{WF}(\alpha)$. By Proposition \ref{prop}, we infer that $U\vdash\mathsf{WF}(\alpha)$, whence $\mathsf{otyp}(\alpha) <|U|_{\mathsf{WF}}$.


\emph{Right to left:} Assume that $|T|_{\mathsf{WF}} \leq |U|_{\mathsf{WF}}$. Let $T\vdash^{\Sigma^1_1}\varphi$ for $\varphi\in\Pi^1_1$. That is, for some true $\Sigma^1_1$ $\psi$, $T+ \psi \vdash \varphi$. Then $T\vdash \psi \to \varphi$; note that $\psi\to\varphi$ is a $\Pi^1_1$ sentence.

Thus, by Theorem \ref{kleene-brouwer}, we have $\mathsf{ACA}_0\vdash (\psi\to\varphi) \leftrightarrow \mathsf{WF}(\alpha)$ for some primitive recursive $\alpha$. So $T\vdash\mathsf{WF}(\alpha)$. Thus, $\mathsf{otyp}(\alpha)<|T|_{\mathsf{WF}}$. By the assumption, $\mathsf{otyp}(\alpha)<|U|_{\mathsf{WF}}$. By Lemma \ref{embeddings}, $U\vdash^{\Sigma^1_1}\mathsf{WF}(\alpha)$, whence $U\vdash^{\Sigma^1_1}\psi\to\varphi$. Thus, $U+\psi\vdash^{\Sigma^1_1} \varphi$ and finally $U\vdash^{\Sigma^1_1}\varphi$.
\end{proof}

\subsection{The Second Equivalence}

In this subsection we will prove the second equivalence. After the author posted a preprint \cite{walsh2022robust} of this article online, Fedor Pakhomov found an alternative proof of the right-to-left direction. Pakhomov's proof is similar in some ways to the original proof but simpler technically since it avoids a detour through $\Sigma^1_1\text{-}\mathsf{AC}_0$. We present Pakhomov's proof here, with his permission.

Before presenting the proof, we will cover two small lemmas. The first concerns the relationship between the $\Sigma^1_1$ formulas and the related class of essentially $\Sigma^1_1$ formulas.

\begin{definition}
The class of \emph{essentially $\Sigma^1_1$ formulas} is the smallest class of formulas that contains all arithmetical formulas and is closed under conjunction, disjunction, universal number quantification, existential number quantification, and existential set quantification.
\end{definition}

\begin{lemma}\label{essentially}
For any essentially $\Sigma^1_1$ formula $\varphi$, there is a $\Sigma^1_1$ formula $\varphi'$ with the same free variables, such that:
\begin{enumerate}
    \item $\Sigma^1_1\text{-}\mathsf{AC}_0\vdash \varphi \to \varphi'$
    \item $\mathsf{ACA}_0\vdash \varphi'\to\varphi$
\end{enumerate}
\end{lemma}

For a proof of Lemma \ref{essentially}, we refer the reader to [Simpson \cite{simpson2009subsystems} Lemma VIII.6.2]. The next lemma we cover concerns the relationship between uniform $\Pi^1_1$-reflection and correctness about well-foundedness within $\mathsf{ACA}_0$.

\begin{lemma}\label{pwf}
$\mathsf{ACA}_0\vdash \forall \gamma \in \mathsf{PrimRec}\Big( \mathsf{Pr}_T\big(\mathsf{WF}(\gamma)\big)\to\mathsf{WF}(\gamma)\Big) \to \mathsf{RFN}_{\Pi^1_1}(T)$.
\end{lemma}

\begin{proof}
We work with a single-variable schema formalization of $\mathsf{RFN}_{\Pi^1_1}(T)$. That is, it suffices to show that $\mathsf{ACA}_0 + \forall \gamma \in \mathsf{PrimRec}\Big( \mathsf{Pr}_T\big(\mathsf{WF}(\gamma)\big)\to\mathsf{WF}(\gamma)\Big)$ proves each instance of the following schema:
$$\forall x \Big( \mathsf{Pr}_T\big(\varphi({x})\big) \to \varphi({x})\Big) \quad \text{ for $\varphi({x}) \in \Pi^1_1$.}$$
Let $\varphi(x)$ be a $\Pi^1_1$ formula. \textbf{\emph{We reason in}} $$\mathsf{ACA}_0 + \forall \gamma \in \mathsf{PrimRec}\Big( \mathsf{Pr}_T\big(\mathsf{WF}(\gamma)\big)\to \mathsf{WF}(\gamma)\Big):$$ Let $n$ be such that $T$ proves $\varphi(n)$. From Theorem \ref{kb-uniform}, we infer that $T\vdash \varphi(n) \leftrightarrow \mathsf{WF}(\beta_n)$. So, by our assumption that all $T$-provably well-founded primitive recursive linear orders are well-founded, $\beta_n$ is well-founded.

Now by Theorem \ref{kb-uniform}, we also infer that $\varphi(n)$ if and only if $\mathsf{WF}(\beta_n)$. So $\varphi(n)$.
\end{proof}

Now we are ready to present the proof of the second equivalence.

\begin{lemma}\label{hard}
For all arithmetically definable $\Pi^1_1$-sound $T$ and $U$ extending $\mathsf{ACA}_0$:
$$ T\leq_{\mathsf{WF}}U \Longleftrightarrow T\leq_{\mathsf{RFN}_{\Pi^1_1}}^{\Sigma^1_1}U.$$
\end{lemma}

\begin{proof}
\emph{Left to right:} Let $\mathsf{PrimRec}$ be an arithmetic definition of the primitive recursive linear orders. Since $|T|_{\mathsf{WF}}\leq |U|_{\mathsf{WF}}$, by Lemma \ref{embeddings}, for any primitive recursive $\gamma$ such that $T\vdash \mathsf{WF}(\gamma)$, $U\vdash^{\Sigma^1_1}\mathsf{WF}(\gamma)$. That is, the following sentence $\theta$ is true:
$$\theta:= \; \forall \gamma \in \mathsf{PrimRec} \Big(\mathsf{Pr}_T\big(\mathsf{WF}(\gamma)\big) \to \mathsf{Pr}^{\Sigma^1_1}_U\big(\mathsf{WF}(\gamma)\big) \Big).$$
Since $\mathsf{ACA}_0$ proves Proposition \ref{prop}, we infer that:
$$\mathsf{ACA}_0 \vdash \forall \gamma \in \mathsf{PrimRec} \Big(\mathsf{Pr}^{\Sigma^1_1}_U\big(\mathsf{WF}(\gamma)\big) \to \exists \delta \; \exists f \big( \mathsf{Emb}(f,\gamma,\delta) \wedge \mathsf{Pr}_U\big(\mathsf{WF}(\delta)\big) \big)\Big).$$
It follows that:
\begin{equation}\label{complicated}
  \small{  \mathsf{ACA}_0 +\theta \vdash \forall \gamma \in \mathsf{PrimRec} \Big(\mathsf{Pr}_T\big(\mathsf{WF}(\gamma)\big) \to \exists \delta \; \exists f \big( \mathsf{Emb}(f,\gamma,\delta) \wedge \mathsf{Pr}_U\big(\mathsf{WF}(\delta)\big) \big) \Big).}
\end{equation}

\begin{claim}
$\mathsf{ACA}_0+ \theta \vdash \mathsf{RFN}_{\Pi^1_1}(U)\to \mathsf{RFN}_{\Pi^1_1}(T)$
\end{claim}

To establish the claim, \textbf{we reason in $\mathsf{ACA}_0+ \theta$}: Assume that $\mathsf{RFN}_{\Pi^1_1}(U)$. By Lemma \ref{pwf}, it suffices to show that $$\forall \gamma \in \mathsf{PrimRec}\Big( \mathsf{Pr}_T\big(\mathsf{WF}(\gamma)\big)\to\mathsf{WF}(\gamma)\Big).$$
So let $\gamma\in \mathsf{PrimRec}$ and assume that $T\vdash \mathsf{WF}(\gamma)$. By \ref{complicated}, there is some $\delta$ such that $\gamma$ embeds into $\delta$ and $U\vdash \mathsf{WF}(\delta)$. Since $\mathsf{RFN}_{\Pi^1_1}(U)$, we infer that $\mathsf{WF}(\delta)$. Since $\gamma$ embeds into $\delta$, we infer that $\mathsf{WF}(\gamma)$.

This establishes the claim. So now we will go back to \textbf{reasoning externally}.

We are not quite done, since the existential set quantifier in $\theta$ occurs within the scope of a number quantifier. However, that existential set quantifier does occur positively. So $\theta$ is an essentially $\Sigma^1_1$ formula. By Lemma \ref{essentially}, there is a $\Sigma^1_1$ formula $\eta$ that is such that $\mathsf{ACA}_0\vdash \eta \to \theta$ and such that $\Sigma^1_1\text{-}\mathsf{AC}_0 \vdash \eta \leftrightarrow \theta$. Thus:
$$\mathsf{ACA}_0 +\eta \vdash \mathsf{RFN}_{\Pi^1_1}(U)\to  
\mathsf{RFN}_{\Pi^1_1}(T).$$


Since $\Sigma^1_1\text{-}\mathsf{AC}_0 \vdash \eta \leftrightarrow \theta$, $\eta$ is true $\Sigma^1_1$, so we infer that:
$$\mathsf{ACA}_0 \vdash^{\Sigma^1_1} \mathsf{RFN}_{\Pi^1_1}(U)\to 
\mathsf{RFN}_{\Pi^1_1}(T).$$

\emph{Right to left:} Assume for a contradiction that $T\leq_{\mathsf{RFN}_{\Pi^1_1}}^{\Sigma^1_1}U$ but $U<_{\mathsf{WF}}T$. Choose some $\alpha\in\mathsf{PrimRec}$ such that $|\alpha|\geq |U|_{\mathsf{WF}}$ and $T\vdash\mathsf{WF}(\alpha)$. Consider also the true $\Sigma^1_1$ sentence $F$ where:
\begin{flalign*}
F:=&\text{ ``There is $X$ such that for any $\beta\in\mathsf{PrimRec}$: if $U$ proves $\mathsf{WF}(\beta)$,}\\
&\text{ then there is $i\in\mathbb{N}$ such that $(X)_i$ encodes an embedding of $\beta$ into $\alpha$.''}
\end{flalign*}

\begin{claim}
$T+F\vdash \mathsf{RFN}_{\Pi^1_1}(U).$
\end{claim}

By Lemma \ref{pwf}, to establish the claim it suffices to show that $$T+F\vdash \forall \gamma \in \mathsf{PrimRec}\Big( \mathsf{Pr}_U\big(\mathsf{WF}(\gamma)\big)\to\mathsf{WF}(\gamma)\Big).$$ \textbf{\emph{We reason in}} $T+F$: Let $\beta\in\mathsf{PrimRec}$ be $U$-provably well-founded. By $F$, we infer that $\beta$ embeds into $\mathsf{WF}(\alpha)$. Since $\mathsf{WF}(\alpha)$, we infer that $\mathsf{WF}(\beta)$.

This establishes the claim. So now we will go back to \textbf{reasoning externally}.

By Lemma \ref{lemma}, we have:
   $$ \mathsf{ACA}_0+F\vdash \mathsf{RFN}_{\Pi^1_1}(T) \to \mathsf{RFN}_{\Pi^1_1}(T+F).$$
Combining this with the claim, we have:
\begin{equation}\label{from_F}
    \mathsf{ACA}_0+F\vdash \mathsf{RFN}_{\Pi^1_1}(T) \to \mathsf{RFN}_{\Pi^1_1}(T+\mathsf{RFN}_{\Pi^1_1}(U)).
\end{equation}

Since $T\leq_{\mathsf{RFN}_{\Pi^1_1}}^{\Sigma^1_1}U$, for some true $\Sigma^1_1$ sentence $G$ we have:
\begin{equation}\label{from_G}
    \mathsf{ACA}_0 + G\vdash \mathsf{RFN}_{\Pi^1_1}(U) \to \mathsf{RFN}_{\Pi^1_1}(T).
\end{equation}

Combining \ref{from_F} and \ref{from_G}, we infer that:
$$\mathsf{ACA}_0+F+G\vdash\mathsf{RFN}_{\Pi^1_1}(U)\to\mathsf{RFN}_{\Pi^1_1}(T+\mathsf{RFN}_{\Pi^1_1}(U)).$$
Since $T$ contains $\mathsf{ACA}_0$:
$$\mathsf{ACA}_0+F+G+\mathsf{RFN}_{\Pi^1_1}(U)\vdash\mathsf{RFN}_{\Pi^1_1}(\mathsf{ACA}_0+\mathsf{RFN}_{\Pi^1_1}(U)).$$ 
By Lemma \ref{lemma}:
$$\mathsf{ACA}_0+F+G+\mathsf{RFN}_{\Pi^1_1}(U)\vdash\mathsf{RFN}_{\Pi^1_1}\big(\mathsf{ACA}_0+F+G+\mathsf{RFN}_{\Pi^1_1}(U)\big).$$ 

So $\mathsf{ACA}_0+F+G+\mathsf{RFN}_{\Pi^1_1}(U)$ is inconsistent by G\"{o}del's second incompleteness theorem. Yet $\mathsf{ACA}_0+F+G+\mathsf{RFN}_{\Pi^1_1}(U)$ is axiomatized by true sentences. Contradiction.
\end{proof}

\subsection{The Full Equivalence}

Our second main theorem follows immediately from Lemma \ref{easy} and Lemma \ref{hard}. Note that the following is a restatement of Theorem \ref{main-intro}:
\begin{theorem*}
For $\Pi^1_1$-sound arithmetically definable $T$ and $U$ extending $\mathsf{ACA}_0:$
$$T\subseteq_{\Pi^1_1}^{\Sigma^1_1}U \Longleftrightarrow T\leq_{\mathsf{WF}} U  \Longleftrightarrow T\leq^{\Sigma^1_1}_{\mathsf{RFN_{\Pi^1_1}}} U.$$
\end{theorem*}

Since the ordinals are well-ordered, Corollary \ref{cor-intro} (restated here for convenience) immediately follows:

\begin{corollary*}
The relations $\subseteq_{\Pi^1_1}^{\Sigma^1_1}$ and $\leq^{\Sigma^1_1}_{\mathsf{RFN}_{\Pi^1_1}}$ pre-well-order the $\Pi^1_1$-sound arithmetically definable extensions of $\mathsf{ACA}_0$.
\end{corollary*}

\subsection{A Negative Result}\label{negative}

Note that Lemma \ref{easy} is stated for $\Sigma^1_1$-definable theories but Lemma \ref{hard} is stated only for arithmetically definable theories. If we could prove Lemma \ref{hard} for all $\Sigma^1_1$-definable theories, we would thereby strengthen Theorem \ref{main-intro}. Thus, it is worth pointing out where our proof of Lemma \ref{hard} would break down if we assumed only that $T$ and $U$ are $\Sigma^1_1$-definable. Look specifically at the left-to-right direction. Note that the sentence $\theta$ has ``$T$ proves $\mathsf{WF}(\gamma)$'' in the antecedent of a conditional. If $T$ is merely $\Sigma^1_1$-definable, then $\theta$ will not even be essentially $\Sigma^1_1$, since $\mathsf{Pr}_T\big(\mathsf{WF}(\gamma)\big)$ occurs negatively in $\theta$. So our $\Sigma^1_1$ oracle will not give us access to the sentence $\theta$. That is, from the conclusion $$\mathsf{ACA}_0+\theta \vdash \mathsf{RFN}_{\Pi^1_1}(U)\to \mathsf{RFN}_{\Pi^1_1}(T)$$ we cannot infer that $\mathsf{ACA}\vdash^{\Sigma^1_1} \mathsf{RFN}_{\Pi^1_1}(U)\to \mathsf{RFN}_{\Pi^1_1}(T).$

In fact, it is not possible to strengthen Lemma \ref{hard} to cover all $\Sigma^1_1$-definable theories. In the next subsection we will prove a lemma that we will use to this end, namely, a version of the Kreisel--L\'{e}vy unboundedness theorem. In the following subsection we will show that Lemma \ref{hard} cannot be strengthened to cover all $\Sigma^1_1$-definable theories. 

\subsubsection{An unboundedness result}

In this subsubsection we prove an analogue of the Kreisel--L\'{e}vy unboundedness theorem from \cite{kreisel1968reflection} \textsection 8; for a modern presentation of the Kreisel--L\'{e}vy theorem, see \cite{beklemishev2005reflection} \textsection 2.4. We derive our result from an analogue of the second incompleteness theorem proved in \cite{walsh2021incompleteness}:
\begin{theorem}[W]\label{incompleteness}
If $T$ is a $\Sigma^1_1$-definable and $\Pi^1_1$-sound extension of $\Sigma^1_1\text{-}\mathsf{AC}_0$ then $T\nvdash \mathsf{RFN}_{\Pi^1_1}(T)$.
\end{theorem}

Here is our version of the unboundedness theorem:

\begin{lemma}\label{unboundedness}
Let $T$ be a $\Sigma^1_1$-definable $\Pi^1_1$-sound extension of $\Sigma^1_1\text{-}\mathsf{AC}_0$. Then no extension of $T$ by a true $\Sigma^1_1$-sentence proves $\mathsf{RFN}_{\Pi^1_1}(T)$.
\end{lemma}

\begin{proof}
Suppose $T+\varphi\vdash \mathsf{RFN}_{\Pi^1_1}(T)$ where $\varphi$ is $\Sigma^1_1$.

To see that $T+\varphi\vdash \mathsf{RFN}_{\Pi^1_1}(T+\varphi)$, we reason in $T+\varphi$.

\emph{Reasoning in $T+\varphi$:} Let $\psi$ be $\Pi^1_1$ such that $T+\varphi \vdash \psi$. Then $T\vdash \varphi \to \psi$. Note that $\varphi \to \psi$ is a $\Pi^1_1$ sentence. Thus, $\varphi\to\psi$ is true (since $T+\varphi\vdash\mathsf{RFN}_{\Pi^1_1}(T)$). But $\varphi$ is also true. So $\psi$ is true too.

So $T+\varphi \vdash \mathsf{RFN}_{\Pi^1_1}(T+\varphi)$. But then $T+\varphi$ is not $\Pi^1_1$-sound by Theorem \ref{incompleteness}. This contradicts Lemma \ref{lemma}.
\end{proof}

\subsubsection{Conservation}

Let's collect one more result before continuing. First, we recall the Barwise--Schlipf conservation theorem \cite{barwise1975recursively}:

\begin{theorem}[Barwise--Schlipf]\label{barwise}
$\Sigma^1_1\text{-}\mathsf{AC}_0$ is $\Pi^1_2$-conservative over $\mathsf{ACA}_0$.
\end{theorem}

Formalizing the Barwise--Schlipf theorem in $\mathsf{ACA}_0$ yields the following:

\begin{theorem}\label{barwise-internal}
Provably in $\mathsf{ACA}_0$, $\Sigma^1_1\text{-}\mathsf{AC}_0$ is $\Pi^1_2$-conservative over $\mathsf{ACA}_0$.
\end{theorem}

Now we are ready to state the corollary.

\begin{corollary}\label{bar-cor}
Provably in $\mathsf{ACA}_0$, for every $\varphi\in\Pi^1_1$, if $\mathsf{RFN}_{\Pi^1_1}(\mathsf{ACA}_0+\varphi)$, then $\mathsf{RFN}_{\Pi^1_1}(\Sigma^1_1\text{-}\mathsf{AC}_0+\varphi)$.
\end{corollary}

\begin{proof}
Reason in $\mathsf{ACA}_0$. Suppose $\neg \mathsf{RFN}_{\Pi^1_1}(\Sigma^1_1\text{-}\mathsf{AC}_0+\varphi)$, i.e., $\Sigma^1_1\text{-}\mathsf{AC}_0 \vdash \varphi \to \psi$ for some false $\Pi^1_1$ sentence $\psi$. Note that $\varphi \to \psi$ is $\Pi^1_2$. So by Theorem \ref{barwise-internal}, $\mathsf{ACA}_0\vdash \varphi \to \psi$. So  $\neg \mathsf{RFN}_{\Pi^1_1}(\mathsf{ACA}_0+\varphi)$.
\end{proof}

Combining this lemma with the unboundedness lemma yields the following useful fact.

\begin{proposition}\label{unbound-aca}
For any true $\Pi^1_1$ sentence $\varphi$: $$\mathsf{ACA}_0 +\varphi \nvdash^{\Sigma^1_1} \mathsf{RFN}_{\Pi^1_1}\big( \mathsf{ACA}_0 + \varphi \big).$$
\end{proposition}

\begin{proof}
Suppose that: $$\mathsf{ACA}_0 +\varphi \vdash^{\Sigma^1_1} \mathsf{RFN}_{\Pi^1_1}\big( \mathsf{ACA}_0 + \varphi \big).$$
By Corollary \ref{bar-cor}: $$\mathsf{ACA}_0 +\varphi \vdash^{\Sigma^1_1} \mathsf{RFN}_{\Pi^1_1}\big( \Sigma^1_1\text{-}\mathsf{AC}_0 + \varphi \big).$$
Whence:
$$\Sigma^1_1\text{-}\mathsf{AC}_0 +\varphi \vdash^{\Sigma^1_1} \mathsf{RFN}_{\Pi^1_1}\big( \Sigma^1_1\text{-}\mathsf{AC}_0 + \varphi \big).$$
This conclusion contradicts Lemma \ref{unboundedness}.
\end{proof}

\subsection{A negative result}

Now we are ready to see that the second main theorem cannot be strengthened to cover all $\Sigma^1_1$-definable theories. Thanks to Fedor Pakhomov for suggesting the following proof.

\begin{proposition}\label{negative-prop}
There are $\Pi^1_1$-sound $\Sigma^1_1$-definable $T$ and $U$ extending $\mathsf{ACA}_0$ such that:
$$  T\leq_{\mathsf{WF}}U\text{ but }T\not\leq_{\mathsf{RFN}_{\Pi^1_1}}^{\Sigma^1_1}U.$$
\end{proposition}
In short, the reason is that there are $\Sigma^1_1$-definable theories with small proof-theoretic ordinals but arbitrarily strong $\Pi^1_1$-reflection statements. We give a more formal argument here:

\begin{proof}
Let $F$ be a true $\Pi^1_1$ sentence such that
\begin{equation}\label{hypo}
    \mathsf{ACA}_0+\mathsf{RFN}_{\Pi^1_1}(\mathsf{ACA}_0)\nvdash^{\Sigma^1_1}F.
\end{equation} Note that there exists such an $F$ by Proposition \ref{unbound-aca}.

We now consider the theory:
$$T:=\mathsf{ACA}_0 + \{\varphi \in \mathcal{L}_2 \mid F\text{ is false}\}.$$
From the external perspective, we can see that $T$ is equivalent to $\mathsf{ACA}_0$, so its proof-theoretic ordinal is $\varepsilon_0$. So $T\leq_{\mathsf{WF}}\mathsf{ACA}_0.$

Immediately from the definition of $T$, one can see that $T$ is either inconsistent or equal to $\mathsf{ACA}_0$ depending on the truth-value of $F$. So it is easy to see that, within $\mathsf{ACA}_0$, $\mathsf{RFN}_{\Pi^1_1}(T)$ is equivalent to $\mathsf{RFN}_{\Pi^1_1}(\mathsf{ACA}_0) \wedge F$. In particular:
\begin{equation}\label{internal}
    \mathsf{ACA}_0 \vdash \mathsf{RFN}_{\Pi^1_1}(T) \to  F .
\end{equation}

Suppose toward a contradiction that $T\leq_{\mathsf{RFN}_{\Pi^1_1}}^{\Sigma^1_1}\mathsf{ACA}_0$, i.e., that:
\begin{equation}\label{contra}
    \mathsf{ACA}_0 \vdash^{\Sigma^1_1} \mathsf{RFN}_{\Pi^1_1}(\mathsf{ACA}_0) \to \mathsf{RFN}_{\Pi^1_1}(T).
\end{equation}

We then reason as follows:
\begin{flalign*}
\mathsf{ACA}_0 &\vdash \mathsf{RFN}_{\Pi^1_1}(T) \to  F  \text{ by (\ref{internal})};\\
\mathsf{ACA}_0 &\vdash^{\Sigma^1_1} \mathsf{RFN}_{\Pi^1_1}(\mathsf{ACA}_0) \to F \text{ by (\ref{contra})}; \\
\mathsf{ACA}_0 + \mathsf{RFN}_{\Pi^1_1}(\mathsf{ACA}_0) &\vdash^{\Sigma^1_1}  F.
\end{flalign*}
Yet this conclusion contradicts (\ref{hypo}).
\end{proof}

\bibliographystyle{plain}
\bibliography{bibliography}

\end{document}